\documentclass[12pt]{article}

\usepackage{setspace}

\usepackage{color}   
\usepackage{hyperref}
\hypersetup{
	colorlinks=true, 
	linktoc=all,     
	linkcolor=blue,  
	citecolor=red,
}
\usepackage{geometry}                
\usepackage{graphicx}
\usepackage{amssymb}
\usepackage{amsmath}
\usepackage{appendix}
\usepackage{tikz}
\usepackage[numbers]{natbib}

\usetikzlibrary{positioning}
\usetikzlibrary{arrows}

\newtheorem{properties}{Properties}[section]

\newtheorem{theorem}{Theorem}[section]
\newtheorem{corollary}{Corollary}[section]
\newtheorem{lemma}{Lemma}[section]

\newenvironment{proof}{{\noindent\it Proof}\quad}{\hfill $\square$\par} 

\numberwithin{figure}{section}
\numberwithin{equation}{section}

\usepackage{algorithm}
\usepackage{algorithmic}
\usepackage{bm}
\newcommand{\relu}{\mbox{{\rm ReLU}}}

\usepackage{stmaryrd}

\begin{document}

\title{ReLU Deep Neural Networks from \\the Hierarchical Basis Perspective}
\author{Juncai He\footnotemark[1] \quad Lin Li\footnotemark[2]\quad  Jinchao Xu\footnotemark[3]}
\date{} 

\maketitle
\renewcommand{\thefootnote}{\fnsymbol{footnote}} 
\footnotetext[1]{Department of Mathematics, The University of Texas at Austin, Austin, TX 78712, USA (jhe@utexas.edu).} 
\footnotetext[2]{Beijing International Center for Mathematical Research, Peking University,
	Beijing 100871, China (lilin1993@pku.edu.cn). } 
\footnotetext[3]{Department of Mathematics, The Pennsylvania State University, University Park, PA 16802, USA (xu@math.psu.edu).} 

\begin{abstract}
We study ReLU deep neural networks (DNNs) by investigating 
their connections with the hierarchical basis method in finite element methods.
First, we show that the approximation schemes
of ReLU DNNs for $x^2$ and $xy$ are composition 
versions of the hierarchical basis approximation for these two functions. 
Based on this fact, we obtain a geometric interpretation and systematic proof
for the approximation result of ReLU DNNs for polynomials, which
plays an important role in a series of recent exponential 
approximation results of ReLU DNNs.  Through our investigation of 
connections between ReLU DNNs and the hierarchical basis 
approximation for $x^2$ and $xy$, we show that ReLU DNNs with this 
special structure can be applied only to approximate quadratic functions. 
Furthermore, we obtain a concise representation to explicitly reproduce any 
linear finite element function on a two-dimensional uniform mesh by using 
ReLU DNNs with only two hidden layers. 
\end{abstract}


%
%
%
%
%

\section{Introduction}\label{sec:intro}
Models based on artificial neural networks (ANNs) have achieved unexpected success in 
a wide range of machine learning and artificial intelligence fields such as computer vision, 
natural language processing, and reinforcement learning~\cite{lecun2015deep, goodfellow2017deep}. 
Mathematical analysis of ANNs can be carried out using many different approaches, including
via the popular approximation and representation theory of ANNs, 
which is critical to developing a mathematical understanding of these models. 
From the 1990s onward, researchers have studied approximation properties for single
hidden layer neural networks as in \cite{hornik1989multilayer, cybenko1989approximation, jones1992simple, barron1993universal, leshno1993multilayer, ellacott1994aspects, pinkus1999approximation, klusowski2018approximation, siegel2020approximation, xu2020finite, siegel2020high, lu2021learning}. 
However, the results of several recent studies~\cite{krizhevsky2012imagenet, he2016deep, huang2017densely} indicate that compared with neural networks with one hidden layer, 
those with more hidden layers -- referred as ``deep neural networks (DNNs)'' -- can achieve
much better performance.
These results motivate further exploration of
the approximation and expression theory of DNNs. Specifically, it is natural to ask how the depth in 
DNNs affects and contributes to the approximation performance and expressive power.

Generally speaking, DNNs with any non-polynomial activation
function can demonstrate what has come to be known as ``universal approximation ability''
~\cite{leshno1993multilayer, siegel2020approximation}.
However, of the activation functions available, the rectified linear unit~\cite{nair2010rectified} activation function
${\rm ReLU(x)} = \max\{ 0, x\}$ is the most successful and, therefore, also the most popular.  
The main reason for its success is twofold: 
Not only can ReLU DNNs achieve state-of-the-art performance in practice, but they also show
rich mathematical structures and properties in theory. Thus, understanding and interpreting 
the approximation efficacy of ReLU DNNs has become a crucial topic, as evidenced by a growing body
of literature~\cite{montufar2014number,telgarsky2016benefits, poggio2017and, 
	yarotsky2017error, lu2017expressive, e2018exponential, montanelli2019new, montanelli2019deep,
	lu2020deep, guhring2020error, marcati2020exponential}. 
From the viewpoint of representation properties,  two notable results~\cite{montufar2014number,telgarsky2016benefits} 
show that ReLU DNNs can produce complicated continuous piece-wise linear functions 
whose number of linear regions grows as the power of  the number of layers. 
Then, \cite{yarotsky2017error} established the first exponential approximation rate for a general function class through the use of ReLU DNN with limited width, by discovering an elegant
approximation property for the square function $s(x) = x^2$ on $[-1,1]$ employing ReLU DNNs.
In the wake of that discovery, numerous research studies have been published in which various exponential error bounds are developed for classical or modified ReLU DNNs for different function families or measurements.
For example, \cite{lu2017expressive} introduced a new and more uniform network architecture to 
achieve these results in \cite{yarotsky2017error}. Then, \cite{e2018exponential} further
improved the network structure by involving architecture similar to that of ResNet~\cite{he2016deep,he2016identity}, thereby obtaining
an exponential convergence rate for a special class of analytic function. 
\cite{montanelli2019new, montanelli2019deep, lu2020deep} investigated the approximation properties 
on Koborov space, bandlimited functions, and $C^s$ functions by approximating sparse grids, 
truncated Chebyshev series, and Taylor expansions.
A series of results~\cite{opschoor2020deep, opschoor2019exponential, 
	guhring2020error, marcati2020exponential} for different function spaces and norms have been obtained
by combining approximation properties in finite element methods.
Although numerous results have been achieved based on the crucial approximation
result of ReLU DNN for the square function $s(x) = x^2$, 
the field still lacks the in-depth understanding needed to extend this special result.

In this paper, we will first concentrate on understanding and interpreting the approximation 
properties of ReLU DNNs for both the square function $s(x)=x^2$ and 
the multiplication function $m(x,y) = xy$ from the  hierarchical basis perspective~\cite{bungartz2004sparse, griebel2005sparse}. 
Briefly, our main discovery is that the ReLU DNN approximation for $s(x) = x^2$
is actually the hierarchical basis approximation for $x^2$ with a composition scheme. 
This observation provides a completely new viewpoint from which to understand the 
ReLU DNN approximation for $x^2$. At the same time, it will greatly simplify the proof developed 
in~\cite{yarotsky2017error} and lead to a 
more precise error estimation. Based on these results, we prove that only quadratic functions can
be efficiently approximated with a special ReLU DNN architecture. 
Through a further investigation of the ReLU DNN approximation for $m(x,y) = xy$ 
from the hierarchical basis viewpoint, we find a geometric interpretation for the result reported in~\cite{yarotsky2017error} originally derived from the pure algebraic relation.
According to this new understanding and interpretation of the approximation properties of ReLU DNNs for $s(x)=x^2$ and $m(x,y) = xy$ from the viewpoint of hierarchical basis and finite element interpolation, we achieve some more accurate and tighter error estimations about the approximation results of ReLU DNNs for polynomials with multi-variables that are than the results in~\cite{yarotsky2017error, e2018exponential, montanelli2019new, montanelli2019deep}.
Furthermore, we achieve a very unexpected and important result from our study of the connections between ReLU DNNs in relation to the hierarchical basis method: that is, we obtain an explicit and concise formula
to represent any linear finite element function on uniform mesh on $\mathbb{R}^2$ with only two hidden layers, which is generally considered to be not trivial and very complicated as discussed in~\cite{arora2018understanding, he2020relu}.

This paper is organized as follows. In Section~\ref{sec:reludnn_CPWL}, we will introduce some
notations and preliminary results for ReLU DNNs. 
In Section~\ref{sec:HB_DNN}, we account for the approximation 
properties of ReLU DNNs for the square function $s(x) = x^2$ by using the hierarchical basis method and show that only quadratic functions can achieve an exponential approximation rate under a specific ReLU DNN architecture. 
In Section~\ref{sec:expressive}, we draw on these same techniques to study the multiplication function $m(x,y) = xy$ in order to meet two related goals: (i) to capture some further approximation and expressive properties of ReLU DNNs for $m(x,y)$, and (ii) to obtain tighter error estimations for approximating multi-variable polynomials. 
In Section~\ref{sec:2DFEMandDNN}, we present a detailed account of our unexpected discovery from the previous investigation of $m(x,y)$. Specifically, we show how to represent a finite element function with uniform mesh on 2D by using ReLU DNNs with only two hidden layers.
In Section~\ref{sec:conclusion}, we offer some concluding remarks.

\section{Preliminary results of ReLU DNNs}\label{sec:reludnn_CPWL}
In this section, we will briefly discuss the definition and properties of
the DNNs generated by using the ReLU activation function.
\subsection{The classical DNN with ReLU activation function}
First, we introduce the general plain DNN function
$f: \mathbb{R}^d \to \mathbb{R}^c$ with L hidden layers by
\begin{equation}\label{eq:DNNdef_J}
	\begin{cases}
		f^{\ell}(x) &= \sigma \circ  \theta^{\ell} (f^{\ell-1}(x)) \quad \ell = 1:L \\
		f(x) &= \theta^{L+1}(f^L(x))
	\end{cases},
\end{equation}
with $f^{0}(x) = x$ where 
\begin{equation}\label{thetamap}
	\theta^{\ell}:\mathbb{R}^{n_{\ell-1}}\to\mathbb{R}^{n_{\ell}} ,
\end{equation}
is the (vector) linear function defined by
\begin{equation}\label{key}
	\theta^\ell(x)=W^\ell x+b^\ell,
\end{equation}
where
$W^\ell=(w^\ell_{ij})\in\mathbb{R}^{n_{\ell}\times n_{\ell-1}}$, $b^\ell \in\mathbb{R}^{n_{\ell}}$.

In the above definition, we denote the nonlinear activation function as 
\begin{equation}\label{sigma}
	\sigma: \mathbb{R} \to \mathbb{R}.
\end{equation} 
By applying the function to each component, we can extend this
naturally to 
\begin{equation}\label{key}
	\sigma:\mathbb R^{n}\mapsto \mathbb R^{n}.
\end{equation}
Given $d, L\in\mathbb{N}^+$ and  
\begin{equation}\label{key}
	n_{1:L}= (n_1,\dots,n_{L}) \in (\mathbb{N}^+)^{L} \mbox{ with }n_0=d, n_{L+1}=1, 
\end{equation}
the general DNN from $\mathbb{R}^d$ to $\mathbb{R}$ defined above  is denoted by
\begin{equation}\label{key}
	\mathcal N_{n_{1:L}} := \{ f :  f(x) = \theta^{L+1}\circ \sigma \circ \theta^{L} \circ \sigma \cdots \circ \theta^2 \circ \sigma \circ \theta^1(x)\}.
\end{equation}
A DNN of this kind, referred to as an $(L+1)$-layer DNN, is said
to have $L$ hidden layers. 
Unless stated otherwise, the term ``layer'' should always
be taken to mean ``hidden layer'' in the rest of this paper.
The size of this DNN is $n_1+\cdots+n_L$.  For the activation functions, in this paper, 
we focus on considering a special activation function, known as the {``rectified linear unit''}
(ReLU), and defined as $\relu: \mathbb R\mapsto \mathbb R$,
\begin{equation}
	\label{relu}
	\relu(x)=\max(0,x), \quad x\in\mathbb{R}. 
\end{equation}
A ReLU DNN with $L$ hidden layers might be written as
\begin{equation}
	\label{relu-dnn}
	f(x) = \theta^{L+1}\circ \relu \circ \theta^{L} \circ \relu \cdots \circ \theta^2 \circ \relu \circ \theta^1(x).
\end{equation}
Furthermore, if $n_{1:L}  = (N, N, \cdots, N)$ we denote
\begin{equation}\label{key}
	\mathcal N_{n_{1:L}} := \mathcal N_L^N ,
\end{equation}
for simplicity.

\subsection{A class of ResNet-type ReLU DNNs}
In this subsection, we will introduce a class of specific ReLU DNNs with some
special short-cut connections, as described in \cite{yarotsky2017error, e2018exponential}. 

Following an idea similar to that above, for any $n_{1:L}= (n_1,\dots,n_{L}) \in (\mathbb{N}^+)^{L} \mbox{ with }n_0=d, n_{L+1}=1$, 
we define the next ReLU DNN with special short-cut connections:
\begin{equation}\label{def:ReLUDNN2}
	\begin{cases}
		f^{1}(x) &= \sigma \circ \widehat\theta^1 (x) \\
		f^{\ell}(x) &= \sigma \circ  \widehat\theta^{\ell} ([x, f^{\ell-1}(x)]) \quad \ell = 2:L \\
		f(x) &= \widehat\theta^{L+1}( [x, f^1, \cdots, f^L] ) \\
	\end{cases}.
\end{equation}
Here, we have $[x, f^{\ell-1}(x)] \in \mathbb{R}^{d + n_{\ell-1}}$ for 
$\ell = 2:L$, $[x, f^1, \cdots, f^L] \in \mathbb{R}^{\sum_{i=0}^L n_i}$ 
and $f^\ell \in \mathbb{R}^{n_\ell}$ for $\ell=1:L$.
In addition, we note that $\widehat\theta^\ell$ is an affine mapping to $\mathbb{R}^{n_\ell}$.
Furthermore, we denote the above function class as $\widehat{\mathcal N}_{n_{1:L}}$ and
\begin{equation}\label{key}
	\widehat{\mathcal N}_{L}^{N} = \widehat{\mathcal N}_{n_{1:L}},
\end{equation}
if $n_{1:L} = (N, N, \cdots, N)$. 
It is easy to see that this function class is quite different from the
standard ReLU DNN models. For $\widehat{\mathcal N}_{L}^{N}$, 
it always needs the original input data $x$ to be the input for all hidden layers. In addition, the last
output layer of $\widehat{\mathcal N}_{L}^{N}$ collects outputs from all hidden layers and then make a linear combination. 
Fig.~\ref{fig:N32} shows more intuitive differences between $\mathcal N_{L}^N$ and $\widehat{\mathcal N}_{L}^N$ where $L=3$ and $N=2$.
\begin{figure}[H]
	\centering
	\includegraphics[width=.4\textwidth]{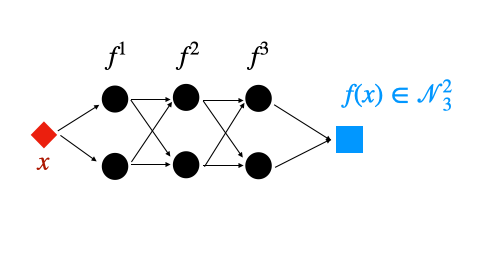}\quad
	\includegraphics[width=.4\textwidth]{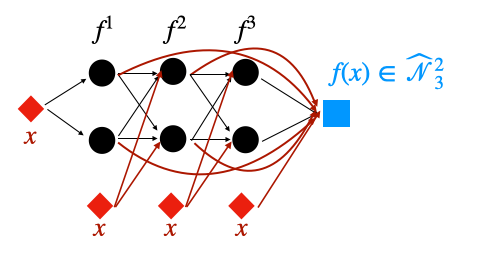}
	\caption{Diagrams for $\mathcal N_{3}^2$ (left) and $\widehat{\mathcal N}_{3}^2$ (right).}
	\label{fig:N32}
\end{figure}

Given the $\relu$ function whereby
\begin{equation}\label{eq:propofRELU}
	x = \relu(x) - \relu(-x),
\end{equation}
we can, however, reconstruct the special ReLU DNN model $\widehat{\mathcal N}_{L}^{N}$
from the standard ReLU DNN model as shown in the following Lemma.

\begin{lemma}(Connection between $\widehat{\mathcal N}_{L}^{N}$ and ${\mathcal N}_{L}^{M}$)\label{lem:HatN_N}
	For the fixed-width case, we have the next connection: 
	\begin{equation}\label{key}
		\widehat{\mathcal N}_{L}^{N} \subset {{\mathcal N}}_{L}^{N + 2(d+1)}.
	\end{equation}
\end{lemma}

\begin{proof}
	Assume that $\widehat f \in \widehat{{\mathcal N}}_{L}^{N}$. This means that there exist $\widehat \theta^{\ell}$ for $\ell = 1:L+1$ such that
	\begin{equation}\label{key}
		\widehat f =  \widehat\theta^{L+1}([x, \widehat f^1, \cdots, \widehat f^L]) = \sum_{\ell=0}^L [\widehat \theta^{L+1}]_{\ell} (\widehat f^\ell),
	\end{equation}
	where $\widehat f^0 = x$, $\widehat f^1 = \sigma\circ \widehat \theta^1 (x)$ and 
	\begin{equation}\label{key}
		\widehat f^\ell = \sigma\circ \widehat \theta^\ell ([x, \widehat f^{\ell-1}]) = \sigma ( [\widehat \theta^\ell ]_{x}(x) + [\widehat \theta^\ell ]_{f}(\widehat f^{\ell-1})),
	\end{equation}
	as defined in \eqref{def:ReLUDNN2}.
	Then, we can construct a function $f \in {{\mathcal N}}_{L}^{N + 2(d+1)}$ such that 
	\begin{equation}\label{key}
		f^1 =  \sigma \circ 
		\begin{pmatrix}
			\widehat \theta^1 (x) \\
			x \\
			-x \\
			[\widehat \theta^{L+1}]_0 (x) \\
			-[\widehat \theta^{L+1}]_0 (x)
		\end{pmatrix} \in \mathbb{R}^{N + 2(d+1)},
	\end{equation}
	and 
	\begin{equation}\label{key}
		f^\ell =  \sigma \circ 
		\begin{pmatrix}
			[\widehat \theta^\ell ]_{f}\left( [f^{\ell-1}]_{1:N}\right) + [\widehat \theta^\ell ]_{x}\left( [f^{\ell-1}]_{N+1:N+d} - [f^{\ell-1}]_{N+d+1:N+2d}\right)  \\
			[f^{\ell-1}]_{N+1:N+d} - [f^{\ell-1}]_{N+d+1:N+2d} \\
			-\left( [f^{\ell-1}]_{N+1:N+d} - [f^{\ell-1}]_{N+d+1:N+2d}\right) \\
			[\widehat\theta^{L+1}]_{\ell-1}([f^{\ell-1}]_{1:N}) + [f^{\ell-1}]_{N+2d+1} - [f^{\ell-1}]_{N+2(d+1)} \\
			-\left([\widehat\theta^{L+1}]_{\ell-1}([f^{\ell-1}]_{1:N}) + [f^{\ell-1}]_{N+2d+1} - [f^{\ell-1}]_{N+2(d+1)}\right)
		\end{pmatrix} \in \mathbb{R}^{N + 2(d+1)},
	\end{equation}
	for all $\ell = 2:L$.
	By recursion, we have
	\begin{equation}\label{key}
		f^\ell =  \sigma \circ 
		\begin{pmatrix}
			\widehat \theta^\ell ([x, \widehat f^{\ell-1}])  \\
			x  \\
			-x  \\
			\sum_{i=0}^{\ell-1}  [\widehat \theta^{L+1}]_{i} (\widehat f^i) \\
			-\left(\sum_{i=0}^{\ell-1}  [\widehat \theta^{L+1}]_{i} (\widehat f^i) \right)
		\end{pmatrix} \in \mathbb{R}^{N + 2(d+1)},
	\end{equation}
	for all $\ell = 2:L$.
	Then, we finish the proof by choosing the following special $\theta^{L+1}$ such that
	\begin{equation}\label{key}
		\begin{aligned}
			f &= \theta^{L+1}(f^L) \\
			&= [\widehat \theta^{L+1}]_L( [f^L]_{1:N}) - [f^L]_{N+2d+1} + [f^L]_{N+2(d+1)} \\
			&= [\widehat \theta^{L+1}]_L( \widehat f^L) + \sigma\left(\sum_{i=0}^{L-1}  [\widehat \theta^{L+1}]_{i} (\widehat f^i) \right) -\sigma\left(-\sum_{i=0}^{L-1}  [\widehat \theta^{L+1}]_{i} (\widehat f^i) \right) \\
			&= \sum_{\ell=0}^L [\widehat \theta^{L+1}]_{\ell} (\widehat f^\ell) = \widehat f.
		\end{aligned}
	\end{equation}
\end{proof}
The properties described next play a crucial role in the structure and approximation properties 
of $\widehat{{\mathcal N}}_{L}^{N} $, which can also be found in \cite{e2018exponential}.
\begin{properties}\label{prop:net class}
	(Addition and composition of $\widehat{{\mathcal N}}_{L}^{N} $)
	\begin{itemize}
		\item Addition
		$$
		\widehat{{\mathcal N}}_{L_1}^{N}  + \widehat{{\mathcal N}}_{L_2}^{N}  \subseteq \widehat{{\mathcal N}}_{L_1 + L_2}^{N},
		$$
		i.e.,  if $f_1\in \widehat{{\mathcal N}}_{L_1}^{N}$ and $f_2\in \widehat{{\mathcal N}}_{L_2}^{N}$, then $f_1+f_2 \in \widehat{{\mathcal N}}_{L_1 + L_2}^{N}$.
		\item Modified composition
		$$
		\widehat{{\mathcal N}}_{L_2}^{N+1} \diamond  \widehat{{\mathcal N}}_{L_1}^{N}\subseteq \widehat{{\mathcal N}}_{L_1 + L_2}^{N+1}.
		$$ 
		Here, if $f_1(x)\in \widehat{{\mathcal N}}_{L_1}^{N}$ and $f_2(x_0,x) \in \widehat{{\mathcal N}}_{L_2}^{N+1}$  with $x_0\in \mathbb{R}$, 
		then the modified composition is defined by
		\begin{equation}\label{key}
			f_2 \diamond f_1 =  f_2(f_1(x),x).
		\end{equation}
		Then, the above property means that
		\begin{equation}\label{key}
			f_2 \diamond f_1  = f_2(f_1(x),x) \in \widehat{{\mathcal N}}_{L_1+L_2}^{N+1}.
		\end{equation}
	\end{itemize}
\end{properties}
\begin{proof}
	For the addition property,  if $f\in \widehat{{\mathcal N}}_{L_1}^{N}$ and $g\in \widehat{{\mathcal N}}_{L_2}^{N}$, we first construct $h \in \widehat{{\mathcal N}}_{L_1 + L_2}^{N}$ with 
	\begin{equation}\label{key}	
		h_\ell(x) = f_\ell(x), \quad \text{for} \quad \ell = 1:L_1.
	\end{equation}
	Because the input of $h_{L_1 + 1}$ includes the original input $x$ directly as
	in the definition of $\widehat{{\mathcal N}}_{L_1 + L_2}^{N}$ in \eqref{def:ReLUDNN2},
	we define
	\begin{equation}\label{key}
		h_{L_1 + 1}(x) = \widehat{\theta}_h^{L_1+1}([x,h_{L_1}(x)]) = \widehat{\theta}_g^1(x) = g_1(x).
	\end{equation}
	Then, we can construct
	\begin{equation}\label{key}
		h_{L_1 + \ell}(x) = g_\ell(x), \quad \text{for} \quad \ell = 2:L_2.
	\end{equation}
	Finally, we obtain
	\begin{equation}\label{key}
		\begin{aligned}
			h(x) &= \widehat{\theta}_{h}^{L_1 + L_2+1} ([x,h_1, \cdots, h_{L_1 + L_2 }]) \\
			&= \widehat{\theta}_{f}^{L_1 +1} ([x,h_1, \cdots, h_{L_1}])  +  \widehat{\theta}_{g}^{L_2+1} ([x,h_{L_1 + 1}, \cdots, h_{L_1 + L_2 }]) \\
			&= \widehat{\theta}_{f}^{L_1 +1} ([x,f_1, \cdots, f_{L_1}])  +  \widehat{\theta}_{g}^{L_2+1} ([x,g_{1}, \cdots, g_{L_2 }]) \\
			&= f(x)+ g(x).
		\end{aligned}
	\end{equation}
	Then, the modified composition property can be directly proved by the definition of $\widehat{{\mathcal N}}_{L}^{N}$ in \eqref{def:ReLUDNN2}.
\end{proof}
From the proofs of Lemma~\ref{lem:HatN_N} and Properties~\ref{prop:net class}, we can see
that the skip connection structure ($f^{\ell}(x) = \sigma \circ  \widehat\theta^{\ell} ([x, f^{\ell-1}(x)])$) in \eqref{def:ReLUDNN2}, which
differs from the skip connection in the classical ResNet~\cite{he2016deep,he2016identity}, is essential to guaranteeing the previous lemma and properties. 

\section{ReLU DNN and hierarchical basis for the function $x^2$}\label{sec:HB_DNN}
In this section, we begin by discussing the hierarchical basis~\cite{griebel2005sparse} and the DNN approximation for the function $s(x) = x^2$. 
Then, we show that  $s(x)=x^2$ is the unique function which can be directly approximated by composition from the hierarchical basis viewpoint.
\subsection{ReLU DNN approximation of $x^2$: A hierarchical basis interpretation}
We consider a set of equidistant grids $\mathcal T_\ell$ for level $\ell \ge 0$ on the unit interval $\bar{\Omega}=[0,1]$ with mesh size $h_\ell = 2^{-\ell}$. 
Here, we denote grid $\mathcal T_\ell$ as
\begin{equation}\label{def:T_ell}
	\mathcal T_\ell := \left\{  \left.  x_{\ell,i} ~\right|~ x_{\ell,i}=ih_\ell,~ 0\le i\le 2^\ell \right\}.
\end{equation} 
At grid point $x_{\ell,i}$, the nodal basis function $\phi_{\ell,i}(x)$ on $\mathcal T_\ell$ is defined as
\begin{equation}
	\phi_{\ell,i}(x):=
	g\left(\frac{x-(i-1)h_\ell}{2h_\ell}\right)\quad x\in \mathbb{R},
\end{equation}
where
\begin{equation}\label{eq:defg}
	g(x) = \begin{cases}
		2x, \quad &x \in \left[0,\frac{1}{2}\right], \\
		2(1-x), \quad &x \in \left(\frac{1}{2},1\right], \\
		0 , \quad &\text{others}.
	\end{cases}
\end{equation}
By definition, we have $\phi_{1,1}(x) = g(x)$. In addition, for $\ell=0$, 
we have these two basis functions,
\begin{equation}\label{key}
	\phi_{0,0} = 1- x \quad \text{and} \quad \phi_{0,1} = x, \quad \text{for } x \in [0,1].
\end{equation}
These basis functions are used to define the piecewise linear function space
\begin{equation}
	V_\ell :=\mbox{span}\{\phi_{\ell,i}: 0 \le i\le 2^\ell\} 
	\subset H^1(0,1), \quad \ell \ge 0.
\end{equation}
Then, let us define the the piecewise linear interpolation on $\mathcal T_\ell$ as:
\begin{equation}\label{eq:inter_I}
	I_\ell u = \sum_{i=0}^{2^\ell}u(x_{\ell,i}) \phi_{\ell,i}. 
\end{equation}
Thus, we have the next hierarchical decomposition for $L \ge 1$:
\begin{equation}
	I_L u =I_0 u +  \sum_{\ell=1}^L  (I_\ell - I_{\ell-1})u =I_0 u +  \sum_{\ell=1}^L \sum_{i\in \mathcal I_\ell}\mu_{\ell,i}\phi_{\ell,i},
\end{equation}
where
\begin{equation}
	\mathcal I_\ell = \{i\in \mathbb{N}~:~1\le i\le 2^\ell-1, ~i ~\text{ is odd} \}.
\end{equation}
Furthermore, we have 
\begin{equation}\label{key}
	(I_\ell - I_{\ell-1})u = \sum_{i\in \mathcal I_\ell}\mu_{\ell,i}\phi_{\ell,i} = \sum_{i\in \mathcal I_\ell} \left( u(x_{\ell,i}) - \frac{1}{2}\left( u(x_{\ell,i-1})  + u(x_{\ell,i+1})  \right) \right)\phi_{\ell,i}.
\end{equation}
The key observation is that there exists a special decomposition form once we take $u(x) = s(x) = x^2$. 
By the hierarchical decomposition above, we have
\begin{equation}
	\begin{aligned}
		(I_{\ell} - I_{\ell-1})s(x) &= \sum_{i\in \mathcal I_\ell}\left( s(x_{\ell,i}) - \frac{1}{2}\left( s(x_{\ell,i-1})  + s(x_{\ell,i+1})  \right) \right)\phi_{\ell,i}(x)\\
		&=\sum_{i\in \mathcal I_\ell}\left(x_{\ell,i}^2 - \frac{1}{2}(x_{\ell,i-1}^2 + x_{\ell,i+1}^2)\right)\phi_{\ell,i}(x)\\
		&= \sum_{i\in \mathcal I_\ell}\left(x_{\ell,i}^2 - \frac{1}{2}\left((x_{\ell,i}-h_\ell)^2 + (x_{\ell,i}+h_\ell)^2\right) \right)\phi_{\ell,i}(x) \\
		&=-h_\ell^2\sum_{i\in \mathcal I_\ell}\phi_{\ell,i}(x).
	\end{aligned}
\end{equation}
Given that $I_0 s = s(0)\phi_{0,0} + s(1)\phi_{0,1} = x$ on $[0,1]$, we have
\begin{equation}
	I_L s(x) = I_0 s(x) + \sum_{\ell=1}^L  (I_\ell - I_{\ell-1})s(x) = x -\sum_{\ell=1}^\ell h^2_\ell \sum_{i\in \mathcal I_\ell}\phi_{\ell,i}(x) = x -\sum_{\ell=1}^\ell h^2_\ell g_\ell(x),
\end{equation}
where
\begin{equation}\label{key}
	g_\ell(x) := \sum_{i\in \mathcal I_\ell}\phi_{\ell,i}(x).
\end{equation}
By definition, we have the following diagram of $g_\ell(x)$:
\begin{figure}[h!]
	\centering
	\includegraphics[width=0.6\textwidth]{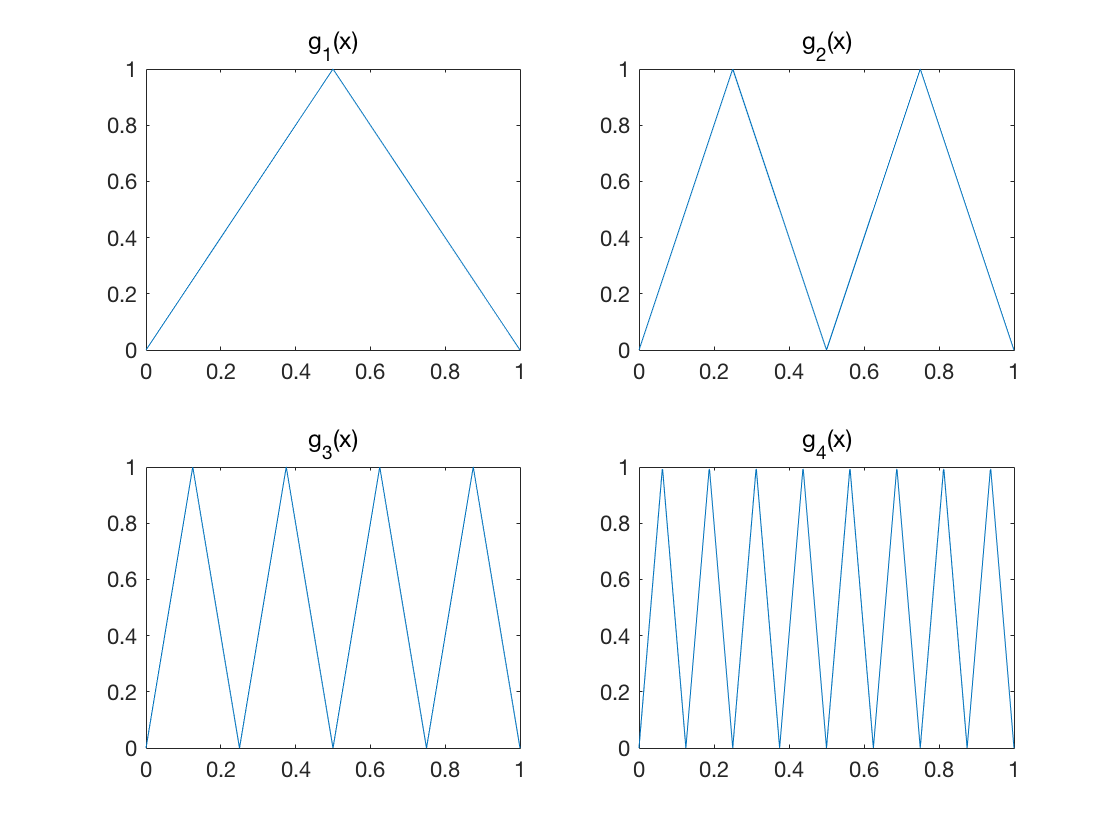}
	\caption{Representations of $g_\ell(x)$ for $\ell=1,2,3,4$.}
	\label{fig:gl}
\end{figure}

The next property plays a crucial role in the connections between ReLU DNNs and the hierarchical basis functions:
\begin{equation}\label{func:glinduction}
	g_\ell(x)
	= g(g_{\ell-1}(x))
\end{equation}
for $\ell=2,...,L$ where $g_1(x) = g(x)$ as defined in~\eqref{eq:defg}.
Based on the connections between the linear finite element functions and the ReLU DNNs as in \cite{he2020relu}, 
we actually have
\begin{equation}\label{def_g}
	g(x) = 2\relu(x) - 4\relu({x-\frac{1}{2}}) + 2\relu(x - 1) \in {{\mathcal N}}_{1}^{3}
\end{equation}
for any $x \in \mathbb{R}$. 
Thus, we have the next key observation,
\begin{equation}
	g_\ell \in {\mathcal N}_\ell^{3},
\end{equation}
because of the composition structure.
Based on the definition of $\widehat{{\mathcal N}}_{\ell}^{3}$, the above relation leads to 
\begin{equation}
	x -\sum_{\ell=1}^L h^2_\ell g_\ell(x) \in \widehat{{\mathcal N}}_{L}^{3}.
\end{equation}
Before we show the approximation result for $s(x)=x^2$ by using an ReLU DNN, 
we recall the interpolation error estimates of $I_L s$ for the $L^{\infty}$ and $W^{1,\infty}$ norms.
\begin{lemma}
	If $s(x) = x^2$, then we have the next approximation results for $I_L$
	\begin{equation}
		\| s -  I_L s \|_{L^{\infty}([0,1])} = 2^{-2(L+1)} \quad \text{and} \quad | s -  I_L s |_{W^{1,\infty}(0,1)} \le 2^{-L}.
	\end{equation}
	
\end{lemma}
Based on the lemma above, we will show a new and concise proof of the next theorem.
\begin{theorem}\label{thm:xsquare1}
	For function $s(x) = x^2$, there exists  $\hat s_L(x) \in \widehat{{\mathcal N}}_{L}^{3}$ such that:
	\begin{equation}\label{eq:x2}
		\| \hat s_L - s \|_{L^\infty([-1,1])} = 2^{-2L} \quad \text{and} \quad | \hat s_L -  s |_{W^{1,\infty}(-1,1)} \le 2^{-(L-1)}.
	\end{equation}
\end{theorem}

\begin{proof}
	We first define
	\begin{equation}\label{key}
		s_{L-1}(x) = I_{L-1} s(x) = x - \sum_{l=1}^{L-1}\frac{g_\ell (x)}{2^{2\ell}} \in\widehat{{\mathcal N}}_{L-1}^{3}.
	\end{equation}
	Then, we construct the function 
	\begin{equation}\label{eq:def_fL}
		\hat s_L (x) = s_{L-1}(|x|) = s_{L-1}(\relu(x)+\relu(-x)) \in \widehat{{\mathcal N}}_{L}^{3}.
	\end{equation}
	Thus, we have
	\begin{equation}
		\hat s_L (x) =
		\begin{cases}
			I_{L-1} s(-x), \quad &x \in [-1,0] \\
			I_{L-1} s(x), \quad &x\in(0,1]
		\end{cases},
	\end{equation}
	where $I_{L-1}$ is the uniform interpolation with mesh size 
	$h_{L-1} = (\frac{1}{2})^{L-1}$ on $[0,1]$ as defined in~\eqref{eq:inter_I}.
	According to the construction of $I_{L-1}s(x)$ and the symmetry of $s(x) = x^2$,
	we have
	\begin{equation}
		\sup_{x\in[-1,0]} |s(x) - I_{L-1}s(-x)| = \sup_{x\in [0,1]}|s(x) - I_{L-1}s(x)|,
	\end{equation}
	which leads to
	\begin{equation}\label{key}
		\|s-\hat s_{L}\|_{L^\infty([-1,0])} = ||s-\hat s_{L}||_{L^\infty([0,1])}.
	\end{equation}
	Finally,  we have the next estimate for the $L^{\infty}$ norm:
	\begin{equation}
		\begin{aligned}
			||s-\hat s_{L}||_{L^\infty([-1,1])} &= \max\left\{ \| s-\hat s_{L}\|_{L^\infty([-1,0])}, \| s-\hat s_{L}\|_{L^\infty([0,1])} \right\}   \\
			&= \| s-\hat s_{L}\|_{L^\infty([0,1])} = || s -I_{L-1} s||_{L^\infty([0,1])} = 2^{-2L}.
		\end{aligned}
	\end{equation}
	Following a similar argument, we can get the result for the $W^{1,\infty}(-1,1)$ norm.
\end{proof}
The result for the $L^{\infty}$ norm was introduced in~\cite{yarotsky2017error} and
then developed in~\cite{e2018exponential} for $\widehat{{\mathcal N}}_{L}^{3}$.
Based on the definition of $\hat s_L(x)$ in \eqref{eq:def_fL}, we have
\begin{equation}\label{eq:prop_fL}
	\begin{aligned}
		\hat s_{L}(x) =~&x_{L-1,i}^2\frac{x_{L-1,i}+h_{L-1}-x}{h_{L-1}} +(x_{L-1,i}+h_{L-1})^2\frac{x-x_{L-1,i}}{h_{L-1}}\\
		=~&(2x_{L-1,i} + h_{L-1})x-x_{L-1,i}(x_{L-1,i}+h_{L-1})
	\end{aligned}
\end{equation}
for any $x\in [x_{L-1,i}, x_{L-1,i}+h_{L-1}]$ and $x_{L-1,i}$ (or $-x_{L-1,i}$)
in grid points $\mathcal T_{L-1}$, as defined in~\eqref{def:T_ell}.
That is, $\hat s_L(x)$ equals the uniform interpolation for
$s(x)$ on $[-1,1]$ with mesh size $h_{L-1} = 2^{1-L}$. 
This account, therefore, provides a geometrical interpretation for $\hat s_L(x)$ and a new understanding of this approximation property.

\subsection{Can the super-approximation result of ReLU DNNs for $x^2$ be extended?}
As discussed, the key connection between the ReLU DNN and hierarchical basis method is the special non-local ``basis'' function 
$g_\ell \in {\mathcal N}_L^3$. Here, a natural idea to extend this connection is to study the representation power
of the space spanned by $\{g_\ell\}_{\ell=1}^\infty$ in $H^1(0,1)$.
For example, if 
\begin{equation}\label{key}
	u = \sum_{\ell=1}^\infty \mu_\ell g_\ell,
\end{equation}
then we can apply the following approximation scheme:
\begin{equation}\label{key}
	u \approx \sum_{\ell=1}^L \mu_\ell g_\ell \in \widehat{{\mathcal N}}_{L}^{3}.
\end{equation}
That is, we can achieve some efficient approximation results using
$\widehat{{\mathcal N}}_{L}^{3}$ for any $u\in U$ if $U = {\rm span}\{g_1, g_2, \cdots\}$ forms a nontrivial and rich subset of $H_0^1(0,1)$.

However, the next theorem shows that $u(x)$ can only be a quadratic function if $u(x)$ is spanned by $\{g_\ell\}_{\ell=1}^\infty$ with certain regularity assumption, for example $u(x) \in C^{3}(0,1)$.
\begin{theorem}
	If $u \in C^{3}(0,1)$ and 
	\begin{equation}\label{key}
		u = I_0 u + \sum_{\ell=1}^\infty \sum_{i \in \mathcal I_\ell}\mu_{\ell} \phi_{\ell,i},
	\end{equation}
	then, $u(x)$ must be a quadratic function.
\end{theorem}

\begin{proof}
	Without loss of generality, let us consider $u(0) = u(1) = 0$. 
	Otherwise, we can subtract the linear interpolation  of the original function with boundary values.
	That is, we have
	\begin{equation}\label{key}
		u = \sum_{\ell=1}^\infty \sum_{i \in \mathcal I_\ell}\mu_{\ell} \phi_{\ell,i}.
	\end{equation}
	According to the $H^1$ orthogonality of $\phi_{\ell,i}$, we can get
	\begin{equation}
		\mu_{\ell}=\int_0^1u'(x)\phi_{\ell,i}(x)dx = \int_0^1-u''(x)\phi_{\ell,i}(x)dx, \text{ for all } \ell \ge 1, i \in \mathcal I_\ell.
	\end{equation}
	Now, let us assume that $u\in C^{3}[0,1]$ is not a quadratic function.
	Thus, there exists $x_0 \in(0,1)$, which satisfies $u'''(x_0)\neq 0$. 
	Without loss of generality, we assume that $u'''(x_0) > 0$, which means 
	that we can find a small interval $(a,b) \subset (0,1)$ such that $u'''(x)>0$ for all $x\in (a,b)$.  
	We can, therefore, find $\ell \ge 1$ large enough and $i\in \mathcal I_\ell$ such that
	$(x_{\ell,i}-h_\ell, x_{\ell,i}+3h_\ell) \subset (a,b)$. 
	This means that $u''(x)$ is monotonically increasing on $(x_{\ell,i}-h_\ell, x_{\ell,i}+3h_\ell)$, such that
	\begin{equation}\label{key}
		\begin{aligned}
			u_\ell &= \int_0^1-u''(x)\phi_{\ell,i}(x)dx \\
			&> \int_0^1-u''(x+2h_\ell)\phi_{\ell,i}(x)dx \\
			&= \int_0^1-u''(x)\phi_{\ell,i+2}(x)dx = u_\ell,
		\end{aligned}
	\end{equation}
	which is a contradiction. Thus, we have $u'''(x) = 0$ on $(0,1)$, which means $u(x)$ 
	has to be a quadratic function.
\end{proof}

\section{ReLU DNN and hierarchical basis for the function $xy$}\label{sec:expressive}
In this section, we will first show a hierarchical basis interpretation for 
the approximation property of ReLU DNN to function as $m(x,y) = xy$ on $[-1,1]^2$, 
which plays a critical role in approximating polynomials by using ReLU DNNs.
Then, we will present the approximation properties of $\widehat{{\mathcal N}}_{L}^{N} $ 
for polynomials on $\mathbb R^d$.

\subsection{ReLU DNN approximation of $xy$: A hierarchical basis interpretation}
A critical step in establishing the connection of polynomials and ReLU DNN is
to approximate $m(x,y) = xy$ in $[-1,1]^2$.
A commonly used approach to achieving this approximation is to
combine  the following algebraic formula
\begin{equation}\label{key}
	m(x,y) := xy = 2\left( \left( \frac{x+y}{2}\right)^2 - \left(\frac{x}{2}\right)^2 - \left( \frac{y}{2}\right)^2 \right)
\end{equation}
with Theorem~\ref{thm:xsquare1}, as in
~\cite{yarotsky2017error,e2018exponential,montanelli2019deep}. 
More precisely, this method approximates $m(x,y)$ by $m_L(x,y)$, which is defined as
\begin{equation}\label{eq:def_phiL}
	\begin{split}
		m_L(x,y) &:= 2\left(\widehat s_L\left(\frac{x+y}{2}\right)-\widehat s_L\left(\frac{x}{2}\right)-\widehat s_L\left(\frac{y}{2}\right)\right)\\
		&= \left(|x+y|-|x|-|y|\right)+2\sum_{\ell = 1}^{L-1}h^2_\ell\left(g_\ell\left(\frac{|x|}{2}\right)+g_\ell\left(\frac{|y|}{2}\right)-g_\ell\left(\frac{|x+y|}{2}\right)\right) \\
		&\in \widehat{{\mathcal N}}_{3L}^3,
	\end{split}
\end{equation}
where $\widehat s_L$ is the approximation of $s(x) = x^2$ in Theorem~\ref{thm:xsquare1}.
The following lemma shows the approximation based on this construction.
\begin{lemma}\label{lemm:varphi_L}
	For $m_L(x,y)$ defined in \eqref{eq:def_phiL}, the next error estimate holds that
	\begin{equation}\label{key}
		\| m - m_L \|_{L^{\infty}([-1,1]^2)} \le 6 \times 2^{-2L} .
	\end{equation}
\end{lemma}
Here, we try to keep the hierarchical decomposition perspective 
to establish a new geometrical understanding for $m_L(x,y)$.
Similar to our presentation of the previous derivation for $s(x)=x^2$, 
we will investigate the hierarchical basis
approximation for $m(x,y) = xy$.
However, unlike the uniform interpolation $I_L$ for $s(x)=x^2$ defined
on $[0,1]$, here we consider the uniform grid interpolation $\Pi_\ell$ for
$m(x,y)$ on $[-1,1]^2$ with mesh size $h_\ell = 2^{-\ell}$ directly. 
In order to differ from notations on 1D, let us define these 2D grid points with the multi-scales
\begin{equation}\label{key}
	\mathcal T^2_\ell := \left\{ \left. (x_{\ell,i}, y_{\ell,j})~\right|~ x_{\ell,i} = (i-2^\ell)h_\ell, y_{\ell,j} = (j-2^\ell)h_\ell, i,j = 0,1,\cdots,2^{\ell+1}\right\}
\end{equation}
for all $\ell \ge 0$. Then for any grid point $(x_{\ell,i}, y_{\ell,j}) \in \mathcal T^2_\ell$ satisfying $i,j\neq 2^{\ell+1}$, 
let us define the upper-right triangle element in terms of grid point $(x_{\ell,i}, y_{\ell,j})$ as
\begin{equation}\label{key}
	T^+_{\ell;i,j} = \left\{ \left. (x,y) \in [x_{\ell,i},x_{\ell,i}+h_\ell]\times[y_{\ell, j},y_{\ell,j}+h_\ell] ~ \right|~ x+y\le x_{\ell,i}+y_{\ell,j}+h_\ell \right\}
\end{equation}
for all $\ell \ge 0$.
Correspondingly, for $i,j\neq 0$ and $\ell \ge 0$, we define the lower-left triangle element in terms of grid point $(x_{\ell,i}, y_{\ell,j})$ as
\begin{equation}\label{key}
	T^-_{\ell;i,j} = \left\{ \left. (x,y) \in [x_{\ell,i}-h_\ell,x_{\ell,i}]\times[y_{\ell, j}-h_\ell,y_{\ell,j}] ~ \right|~ x+y \ge x_{\ell,i}+y_{\ell,j}-h_\ell \right\}.
\end{equation}
According to the representation of the barycentric coordinates, we have the following
interpolation function on each element (see Fig.~\ref{fig:xyelement}):
\begin{equation}\label{eq:def_Pi_L}
	\begin{split}
		\Pi_\ell m(x,y) = &   (x_{\ell,i}+h_\ell)y_{\ell,j}\frac{(x-x_{\ell,i})h_\ell/2}{h_\ell^2/2}+x_{\ell,i}(y_{\ell,j}+h_\ell)\frac{(y-y_{\ell,j})h_\ell/2}{h_\ell^2/2}\\
		&+x_{\ell,i}y_{\ell,j} \frac{h_\ell^2/2-(x-x_{\ell,i})h_\ell/2-(y-y_{\ell,j})h_\ell/2}{h_\ell^2/2}\\
		= & xy_{\ell,j}+yx_{\ell,i}-x_{\ell,i}y_{\ell,j}, \quad (x,y) \in T^+_{\ell;i,j} \cup T^{-}_{\ell;i,j}
	\end{split}
\end{equation}
for all $T^+_{\ell;i,j} \cup T^{-}_{\ell;i,j} \subseteq [-1,1]^2$.

\begin{figure}[h!]
	\centering
	\includegraphics{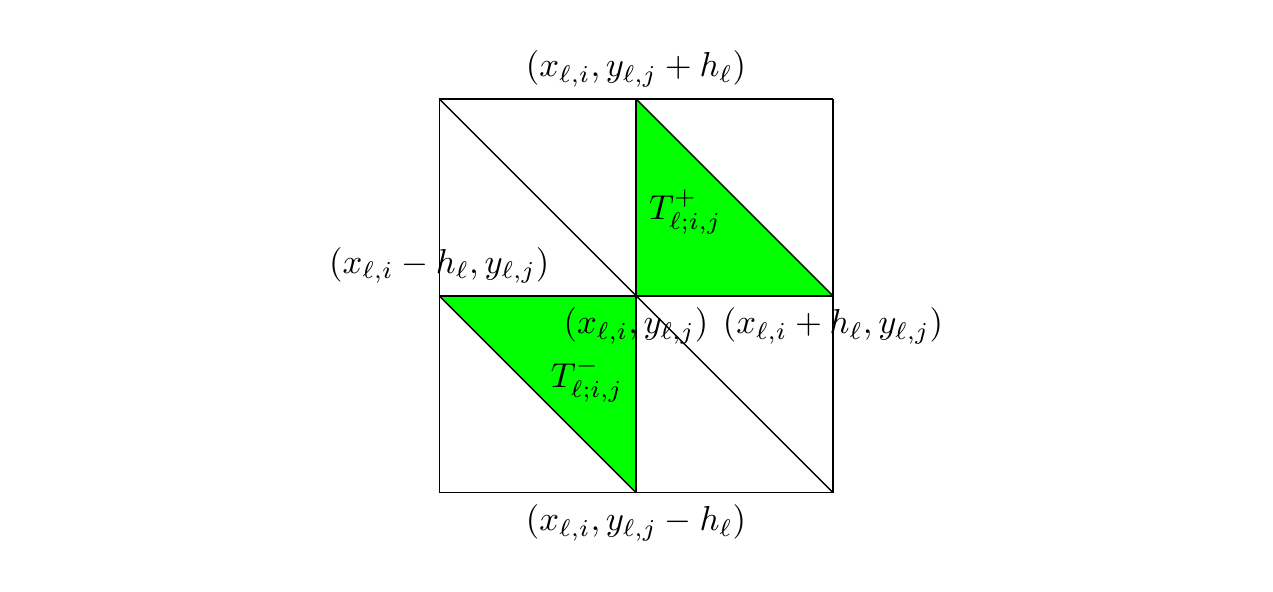}
	\caption{An example of element interpolation on $T^+_{\ell;i,j}\cup T^-_{\ell;i,j}$.}
	\label{fig:xyelement}
\end{figure}


The following theorem describes the connection between $\Pi_Lm$ 
and $m_L$, which gives the geometrical interpretation of $m_L$. 

\begin{theorem}\label{them:xy} For $m(x,y) = xy$, the following identity for 
	the uniform interpolation $\Pi_Lm$ in \eqref{eq:def_Pi_L} and 
	the commonly used approximation scheme $m_{L}$ in \eqref{eq:def_phiL} holds:
	\begin{equation}
		\Pi_Lm (x, y) = m_{L+2}(x,y), \quad (x,y)\in[-1,1]^2.
	\end{equation}
\end{theorem}

\begin{proof}
	Let us compute $m_{L+2}(x,y)$ first. 
	Recall the property of $\widehat s_{L+2}(x)$ in \eqref{eq:prop_fL}. 
	Therefore, for any $x \in [x_{L+1,i}, x_{L+1,i} + h_{L+1}]$ and $(x_{L+1,i}, y_{L+1,j}) \in \mathcal T^2_{L+1}$
	we have
	\begin{equation}\label{eq:interp_fL}
		\widehat s_{L+2}(x) =(2x_{L+1,i} + h_{L+1})x-x_{L+1,i}(x_{L+1,i}+h_{L+1}).
	\end{equation}
	Now, let us consider $(x,y) \in T^+_{L;i,j}$ where for any $i,j$, we have $\frac{x_{L,i}}{2}\le\frac{x}{2}\le\frac{x_{L,i}}{2}+h_{L+1},\frac{y_{L,j}}{2}\le\frac{y}{2}\le\frac{y_{L,j}}{2}+h_{L+1},\frac{x_{L,i}+y_{L,j}}{2}\le\frac{x+y}{2}\le\frac{x_{L,i}+y_{L,j}}{2}+h_{L+1}$. This means that we can apply \eqref{eq:interp_fL} for $\frac{x}{2}$, $\frac{y}{2}$ and $\frac{x+y}{2}$, and get
	\begin{equation}
		\begin{aligned}
			m_{L+2}(x,y) =~&2(\widehat s_{L+2}(\frac{x+y}{2})-\widehat s_{L+2}(\frac{x}{2})-\widehat s_{L+2}(\frac{y}{2}))\\
			=~&2\bigg\{\left((x_{L,i}+y_{L,j}+h_{L+1})\frac{x+y}{2}-\frac{x_{L,i}+y_{L,j}}{2}(\frac{x_{L,i}+y_{L,j}}{2}+h_{L+1})\right)\\
			&-\left((x_{L,i}+h_{L+1})\frac{x}{2}-\frac{x_{L,i}}{2}(\frac{x_{L,i}}{2}+h_{L+1})\right)\\
			&-\left((y_{L,j}+h_{L+1})\frac{y}{2}-\frac{y_{L,j}}{2}(\frac{y_{L,j}}{2}+h_{L+1})\right)\bigg\}\\
			=~&y_{L,j}x+x_{L,i}y-x_{L,i}y_{L,j} \\
			=~&\Pi_Lm(x,y)
		\end{aligned}
	\end{equation}
	for any $(x,y) \in T^+_{L;i,j}$ and $i,j$. With the same argument, we can verify the above relation
	on $T^-_{L;i,j}$ for any $i,j$.
	Thus, $\Pi_L(xy) = m_{L+2}(x,y)$ holds for all $(x,y) \in [-1,1]^2$. 
\end{proof}

This theorem provides a geometrical interpretation of
the commonly used approximation scheme $m_L(x,y)$ in \eqref{eq:def_phiL},
which was originally proposed from the algebraic perspective.

Now let us show the approximation property of ReLU DNN for $m(x,y) = xy$
by utilizing results we have just established. 
Similar to 1D, we present the next lemma for the approximation result of $\Pi_L m$.
\begin{lemma}\label{lem:Pim-m}
	For $m(x,y) = xy$, we have
	\begin{equation}\label{key}
		\|\Pi_L m  - m \|_{L^{\infty}([-1,1]^2)}  = 2^{-2(L+1)}.
	\end{equation}
\end{lemma}
\begin{proof}
	For any $(x,y)\in T^+_{L;i,j} \cup T^-_{L;i,j}$, we have 
	\begin{equation}
		\begin{split}
			|\Pi_Lm(x,y) - m(x,y)| & = |y_{L,j}x+x_{L,i}y-x_{L,i}y_{L,j}-xy|\\
			&= (x-x_{L,i})(y-y_{L,j})\\
			&\le \left(\frac{x+y-x_{L,i}-y_{L,j}}{2}\right)^2\\
			&\le \left(\frac{h_L}{2}\right)^2=2^{-2(L+1)}.
		\end{split}
	\end{equation}
	Thus, we have 
	\begin{equation}\label{key}
		\|\Pi_L m  - m \|_{L^{\infty}([-1,1]^2)}  = 2^{-2(L+1)}.
	\end{equation}
\end{proof}

Based on the above connection and approximation rate of $\Pi_L m$, we have the next lemma.
\begin{lemma}\label{lem:xy}
	For multiplication function $m(x,y) = xy$ on $[-M,M]^2$, we have $m_L (x,y) \in \widehat{{\rm DNN}}_{3L}^{3}$ such that
	\begin{equation}\label{key}
		\|m_L - m \|_{L^\infty([-M,M]^2)} = M^2 2^{-2(L-1)}. 
	\end{equation}
	In addition, $m_L(x,0) = m_L(0,y) = 0  $.
\end{lemma}
\begin{proof}
	Let $\widehat s_L \in \widehat{{\rm DNN}}_{L}^{2}$ be the approximation of $s(x) = x^2$ as presented in Theorem~\ref{thm:xsquare1}.
	We, therefore, generalize the previous construct of $m_L$ as
	\begin{equation}\label{eq:phi_L-M}
		\begin{aligned}
			m_L(x,y) &= M^2\left(2 \widehat s_L \left(\frac{x+y}{2M} \right)-2 \widehat s_L\left(\frac{x}{2M}\right)-2 \widehat s_L \left(\frac{y}{2M}\right) \right) \\
			&= M^2 \Pi_{L-2} m(\frac{x}{M},\frac{y}{M})
		\end{aligned}
	\end{equation}
	for all $(x,y) \in [-M,M]^2$.
	Thus, we have $m_L \in \widehat{{\rm DNN}}_{3L}^{3}$ and 
	\begin{equation}\label{key}
		\begin{aligned}
			\|  m_L  - m \|_{L^{\infty}([-M,M]^2)} &= \sup_{(x,y) \in [-M,M]^2} \left| M^2 \Pi_{L-2} m\left(\frac{x}{M},\frac{y}{M}\right) - M^2 m\left(\frac{x}{M},\frac{y}{M}\right)\right| \\
			&= M^2 \sup_{(x,y) \in [-1,1]^2}  \left| \Pi_{L-2} m(x,y)  - m(x,y) \right|  \\
			&= M^2 2^{-2(L-1)}.
		\end{aligned}
	\end{equation}
\end{proof}

\subsection{Approximation property of $\widehat{{\mathcal N}}_{L}^{N}$  for polynomials}
Here, following our previous results for the ReLU DNN approximation of $m(x,y) = xy$ on $[-1,1]^2$
, we derive the super-approximation properties for polynomials by using ReLU DNN.

First, denote
\begin{equation}\label{eq:x^k}
	\bm x^{\bm k} = x_1^{k_1} x_2^{k_2} \cdots x_d^{k_d}
\end{equation}
as a monomial and $|{\bm k}| = \sum_{i=1}^d k_i$, which is also known as the $\ell_1$ norm
of $\bm k$. 
To prove the approximation properties of ReLU DNN for the monomial, we need the next lemma
to estimate the range of $m_L(x,y)$. This range will become the domain for the next hidden layer
for the composition in DNN.

\begin{lemma}\label{lemm:varphi_infty}
	If $|x|,|y|\le M$, we have 
	\begin{equation}
		|m_L(x,y)| = |2M^2(\widehat s_L(\frac{x+y}{2M})-\widehat s_L(\frac{x}{2M})-\widehat s_L(\frac{y}{2M}))|\le M^2.
	\end{equation}
\end{lemma}
\begin{proof}Based on the relation in~\eqref{eq:phi_L-M}, we have
	\begin{equation}\label{key}
		\begin{aligned}
			\sup_{(x,y) \in [-M,M]^2} |m_L(x,y)| &= M^2\sup_{(x,y) \in [-M,M]^2} \left|  \Pi_{L-2} m(\frac{x}{M},\frac{y}{M}) \right|\\
			&= M^2\sup_{(x,y) \in [-1,1]^2} \left|  \Pi_{L-2} m(x,y) \right|\\
			&= M^2.
		\end{aligned}
	\end{equation}
\end{proof}

\begin{lemma}\label{lemm:relumonomial}
	For any monomial $M_{\bm k}(\bm x) = {\bm x}^{\bm k} = x_1^{k_1} x_2^{k_2} \cdots x_d^{k_d},\bm x\in[-1,1]^d$ 
	with degree $p$, i.e., $|\bm k| = p$, there exists $\widehat M_{\bm k}(\bm x) \in  \widehat{{\mathcal N}}_{3(p-1)L}^{4} $ such that  
	\begin{equation}\label{key}
		\|M_{\bm k}(\bm x) - \widehat M_{\bm k}(\bm x) \|_{L^{\infty}([-1,1]^d)} \le (p-1)\cdot 2^{-2(L-1)}
	\end{equation}
	and $\|\widehat M_{\bm k}(\bm x)\|_{L^{\infty}([-1,1]^d)} \le 1 $.
\end{lemma}

\begin{proof}
	We establish the proof by induction for $p$. Firstly, for $p=2$ and $M_{\bm k} (\bm x) = x_ix_j$, 
	we have $\widehat M_{\bm k} (\bm x) = m_L(x_i, x_j) \in \widehat{{\mathcal N}}_{3L}^{3} \subset \widehat{{\mathcal N}}_{3L}^{4}$ as presented in Lemma~\ref{lem:xy}.
	By induction, let us prove the case with $p+1$ if $p$ holds. Then, for any $M_{\bm k}(\bm x)$ with $|\bm k| = p+1$, we can find an $M_{\tilde {\bm k}}(\bm x)$ with $\tilde {\bm k}=p$
	and $\bm k = \tilde {\bm k} + e_i$ such that $M_{\tilde {\bm k}}(\bm x) \in \widehat{{\mathcal N}}_{3(p-1)L}^{4}$ with 
	\begin{equation}\label{key}
		\|M_{\bm {\tilde k}}(\bm x) - \widehat M_{\bm {\tilde k}}(\bm x) \|_{L^{\infty}([-1,1]^d)} \le (p-1)\cdot 2^{-2(L-1)}
	\end{equation}
	and $\|\widehat M_{\bm {\tilde k}}(\bm x)\|_{L^{\infty}([-1,1]^d)} \le  1$. Then, we have the next construction for $\widehat M_{\bm { k}}(\bm x)$:
	\begin{equation}\label{eq:cons_M_k}
		\widehat M_{\bm { k}}(\bm x) = m_L(x_i, \widehat M_{\bm {\tilde k}}(\bm x)) \in  \widehat{{\mathcal N}}_{3pL}^{4}.
	\end{equation}
	Then, let us check the approximation property:
	\begin{equation}
		\begin{aligned}
			&\|M_{\bm { k}}(\bm x) - \widehat M_{\bm { k}}(\bm x) \|_{L^{\infty}([-1,1]^d)} \\
			&\le \|m(x_i, M_{\bm {\tilde k}} (\bm x)) - m_L(x_i, \widehat M_{\bm {\tilde k}}(\bm x))\|_{L^{\infty}([-1,1]^d)}  \\
			&\le \|m(x_i, M_{\bm {\tilde k}} (\bm x)) - m(x_i, \widehat M_{\bm {\tilde k}}(\bm x))\|_{L^{\infty}([-1,1]^d)}   + \|m(x_i, \widehat M_{\bm {\tilde k}} (\bm x)) - m_L(x_i, \widehat M_{\bm {\tilde k}}(\bm x))\|_{L^{\infty}([-1,1]^d)}  \\
			&\le \|x_i\|_{L^{\infty}([-1,1]^d)} \|M_{\bm {\tilde k}} (\bm x) - \widehat M_{\bm {\tilde k}} (\bm x)\|_{L^{\infty}([-1,1]^d)}+  2^{-2(L-1)} \\
			&\le (p-1)2^{-2(L-1)} +   2^{-2(L-1)} \\
			&\le p2^{-2(L-1)}.
		\end{aligned}.
	\end{equation}
	Then, following Lemma~\ref{lemm:varphi_infty}, we have $\|\widehat M_{\bm { k}}(\bm x) \|_{L^{\infty}([-1,1]^d)} \le 1$ given the construction in~\eqref{eq:cons_M_k} and the induction that 
	$\|\widehat M_{\bm {\tilde k}}(\bm x)\|_{L^{\infty}([-1,1]^d)} \le  1$.
	
\end{proof}

In the end, we have the next approximation result for polynomials in $\mathbb{R}^d$.
\begin{lemma}\label{lemm:reluPolynomial}
	For any degree-$p$ polynomial $P_p(\bm{x}) = \sum_{|\bm{k}|\le p}a_{\bm k}\bm x^{\bm k},\bm x\in [-1,1]^d$, there exists 
	$\widehat f^L_p \in  \widehat{{\mathcal N}}^4_{3\binom{p+d}{d}(p-1)L}$ such that 
	\begin{equation}\label{lem:poly}
		\|P_p - \widehat f^L_p \|_{L^{\infty}([-1,1]^d)} 
		\le (p-1)\cdot 2^{-2(L-1)}\sum_{|\bm k|\le p}|a_{\bm k}|.
	\end{equation}
\end{lemma}

\begin{proof}
	This can be proven by combining Properties~\ref{prop:net class} and Lemma~\ref{lemm:relumonomial}.
\end{proof} 
The above result plays a fundamental role in many studies~\cite{yarotsky2017error, lu2017expressive, e2018exponential, montanelli2019new, montanelli2019deep, lu2020deep, opschoor2019exponential, opschoor2020deep, guhring2020error, marcati2020exponential}. 
In this paper, however, we provide tighter bounds and more concise proofs for Lemmas \ref{lem:Pim-m}--\ref{lemm:reluPolynomial} by applying the connection of ReLU DNNs and FEM in Theorem~\ref{them:xy}.

\section{Representation of  a 2D linear finite element function with only 2 hidden layers}\label{sec:2DFEMandDNN}
In this section, we will present our unexpected discovery pertaining to the
expressive power of ReLU DNNs, which derived from our investigation into ReLU DNNs 
in relation to hierarchical basis methods.

In terms of the approximation of ReLU DNNs for $x^2$,
the most critical point in establishing Theorem~\ref{thm:xsquare1} is that
\begin{equation}\label{key}
	(I_{\ell} - I_{\ell-1})s(x) = -h_\ell^2\sum_{i\in \mathcal I_\ell}\phi_{\ell,i}(x) = -h_\ell^2g_\ell(x) \in {\mathcal N}_\ell^3
\end{equation}
for all $x\in \mathbb{R}$. Thus, a key question arises:
\begin{quote}
	Is there composition structure for $(\Pi_\ell - \Pi_{\ell-1})m (x, y)$?
\end{quote}

Thanks to Theorem~\ref{them:xy}, we have the following 
corollary for the explicit formula for $(\Pi_\ell- \Pi_{\ell-1})m$ on $[-1,1]^2$.
\begin{corollary}
	For any $(x,y) \in [-1,1]^2$, we have
	\begin{equation}\label{eq:Pi-Pi}
		\begin{aligned}
			&(\Pi_\ell- \Pi_{\ell-1})m (x,y) \\
			&= m_{\ell+2}(x,y) - m_{\ell+1}(x,y)  \\
			&= 2h^2_{\ell+1}\left(g_{\ell+1}\left(\frac{|x|}{2}\right)+g_{\ell+1}\left(\frac{|y|}{2}\right)-g_{\ell+1}\left(\frac{|x+y|}{2}\right)\right).
		\end{aligned}
	\end{equation}
	This means that $\psi_\ell \in \widehat{{\mathcal N}}_{3(\ell+2)}^3$ where $\psi_\ell = (\Pi_\ell- \Pi_{\ell-1})m$. Furthermore, 
	it is worth noting that we can actually have $\psi_\ell \in {\mathcal N}_{\ell+2}^9$ since $g_{\ell+1}(\frac{|x+y|}{2}) \in {\mathcal N}_{\ell+2}^3$.
\end{corollary}

By choosing a suitable scale, we have $\|h_{\ell}^{-2}(\Pi_\ell- \Pi_{\ell-1})m\|_{L^{\infty}([-1,1]^2)} = 1$. 
Fig.~\ref{fig:T2T3} shows function graphs of $h_2^{-2}(\Pi_2 - \Pi_{1})m$ and $h_{3}^{-2}(\Pi_3- \Pi_{2})m$ on $[0,1]^2$.
\begin{figure}[h!]
	\centering
	\begin{minipage}[t]{0.43\textwidth}
		\centering
		\includegraphics[width=\textwidth]{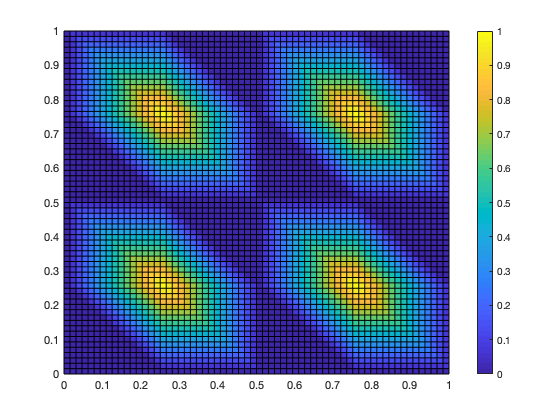}
	\end{minipage}
	\begin{minipage}[t]{0.43\textwidth}
		\centering
		\includegraphics[width=\textwidth]{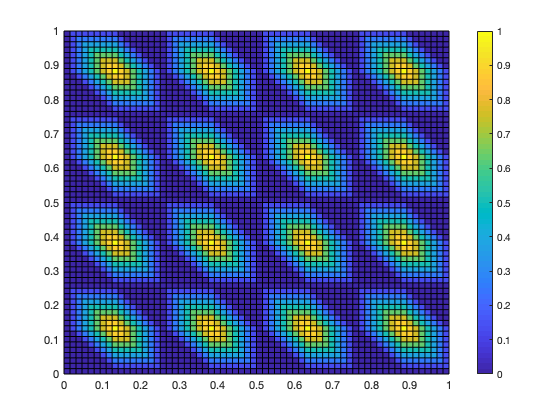}
	\end{minipage}
	\caption{Functions of $h_{\ell}^{-2}(\Pi_\ell- \Pi_{\ell-1})m$ for $\ell = 2$ (left) and $\ell=3$ (right) on $[0,1]^2$.}
	\label{fig:T2T3}
\end{figure}

These graphs and meshes of $(\Pi_\ell - \Pi_{\ell-1})m(x,y)$ give rise to a natural discussion about
the representation theorem of ReLU DNN for the basis function of the 2D linear finite element.
By taking $\ell=1$ with a suitable scale, we have
\begin{equation}\label{eq:varphi_1}
	4(\Pi_1- \Pi_{0})m (x,y) 
	= \frac{1}{2}\left(g_{2}\left(\frac{x}{2}\right)+g_2\left(\frac{y}{2}\right)-g_2\left(\frac{x+y}{2}\right)\right)
\end{equation}
for all $ (x,y) \in [0,1]^2$. 
Here, $4(\Pi_1- \Pi_{0})m (x,y)$ equals the basis function $\varphi(x,y)$ for the 
linear finite element (see left-hand graph of Fig.~\ref{fig:varphi_1}) 
on the uniform mesh on $[0,1]^2$ with mesh size $h_1=\frac{1}{2}$ (see right-hand graph of Fig.~\ref{fig:varphi_1}).
\begin{figure}[h!]
	\begin{minipage}[t]{0.43\linewidth}
		\centering
		\includegraphics[width=\textwidth]{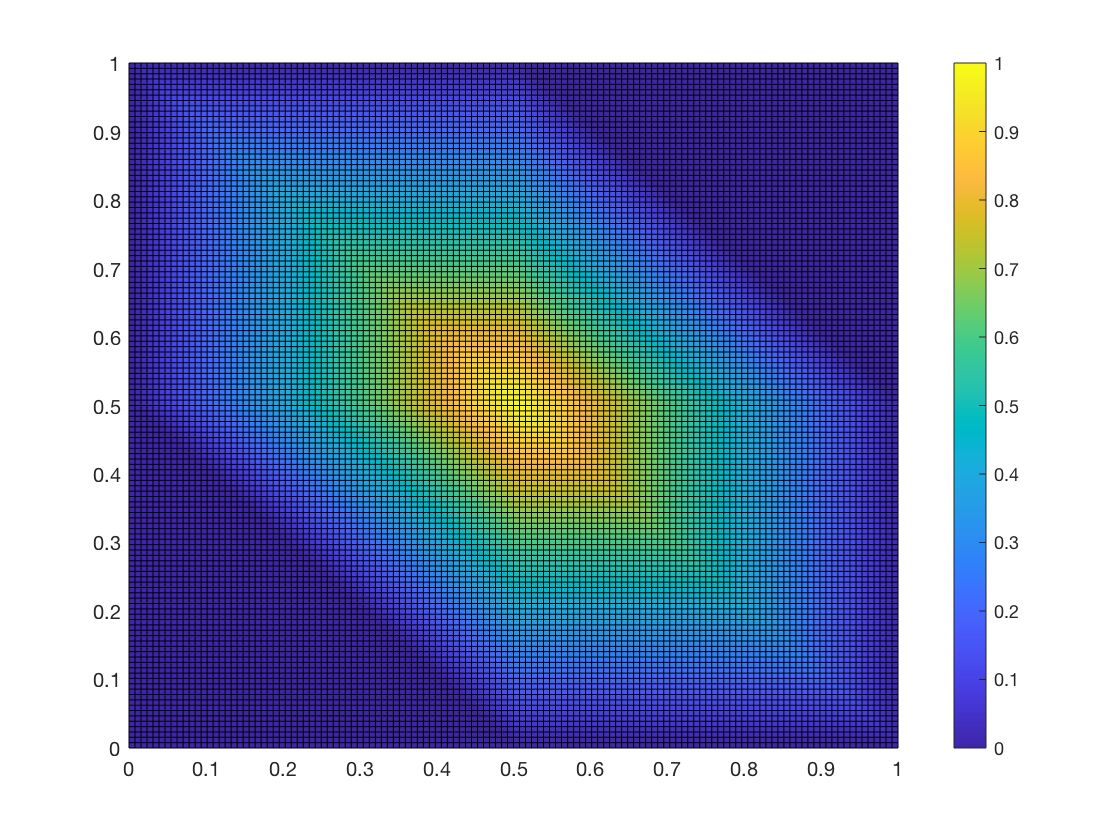}
	\end{minipage}
	\begin{minipage}[t]{0.45\linewidth}
		\centering
		\includegraphics[scale=.5]{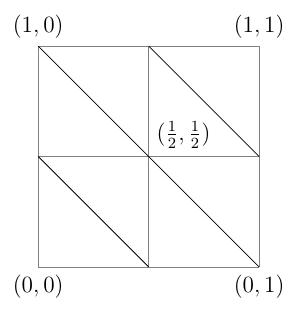}
	\end{minipage}
	\caption{left: $\varphi(x,y)$ on $[0,1]^2$; right:  mesh for $\ell=1$ on $[0,1]^2$.}
	\label{fig:varphi_1}
\end{figure}
Without loss of generality, we can define $\varphi(x,y)$ globally as
\begin{equation}\label{key}
	\varphi(x,y) = \begin{cases}
		4(\Pi_1- \Pi_{0})m (x,y), \quad &(x,y) \in [0,1]^2, \\
		0, \quad &\text{others}.
	\end{cases}
\end{equation}
However, this representation,
\begin{equation}\label{eq:varphi_1-1}
	\varphi(x,y) = \frac{1}{2}\left(g_{2}\left(\frac{x}{2}\right)+g_2\left(\frac{y}{2}\right)-g_2\left(\frac{x+y}{2}\right)\right),
\end{equation}
holds only for $(x,y) \in [0,1]^2$. 
Related to this observation is the simple argument that 
\begin{equation}\label{key}
	\frac{1}{2}\left(g_{2}\left(\frac{x}{2}\right)+g_2\left(\frac{y}{2}\right)-g_2\left(\frac{x+y}{2}\right)\right) = \frac{1}{2} \neq 0
\end{equation}
if $(x,y) = (\frac{3}{2}, \frac{3}{2})$.
On the other hand, on the assumption that the identity in \eqref{eq:varphi_1-1} holds for all $(x,y) \in \mathbb{R}^2$, we can rewrite $g_2(x)$ as
\begin{equation}\label{eq:defg2_2}
	g_2(x) = \sum_{i=0}^4 \alpha_i {\rm ReLU}\left(x-\frac{i}{4}\right),
\end{equation}
for all $x\in \mathbb{R}$ where $(\alpha_0,\alpha_1,\cdots, \alpha_4) = (4, -8, 8, -8, 4)$ based on the relations between the linear finite element functions and ReLU DNNs on 1D~\cite{he2020relu}.
This means that $g_2 \in {\mathcal N}_{1}^5$ and $\varphi(x,y)$ can be represented by a ReLU DNN with only one hidden layer on the entire space $\mathbb{R}^2$.
However, this representation is contradictory to the theorem in~\cite{he2020relu} that the locally supported basis function of the 2D linear finite element cannot be represented globally by ReLU DNN  with just one hidden layer.

Given the global representation of $\varphi(x,y)$, \cite{he2019finite} constructs 
a ReLU DNN with four hidden layers to reproduce $\varphi(x,y)$ explicitly. 
Although \cite{arora2018understanding,he2020relu} show that
a two-hidden-layer ReLU DNN should be able to represent $\varphi(x,y)$ globally on $\mathbb{R}^2$,
the structure of network will be extremely complicated, as well as it requires a 
large number of neurons. Thus, to find a concise formula to 
reproduce $\varphi(x,y)$ on $\mathbb{R}^2$ directly becomes the focus of the inquiry.
Based on the discovery in \eqref{eq:varphi_1-1} and the properties of ReLU DNNs, 
we can construct a ReLU DNN function with two hidden layers to represent $\varphi(x,y)$. 
To simplify the statement of that result, let us first denote the following
${\rm ReLU1}(x)$ function:
\begin{equation}\label{eq:defReLU1}
	{\rm ReLU1}(x)  := {\rm ReLU}(x) - {\rm ReLU}(x-1) \in {\mathcal N}_1^2.
\end{equation}

\begin{lemma}
	The basis function $\varphi(x,y)$ is in  ${\mathcal N}_2^{15}$, more precisely, we have
	\begin{equation}\label{eq:b-DNN2}
		{\varphi}(x,y) = \frac{1}{2}\left(g_{2}\left(\frac{{\rm ReLU1}(x)}{2}\right)+g_2\left(\frac{{\rm ReLU1}(y)}{2}\right)-g_2\left(\frac{{\rm ReLU1}(x)+{\rm ReLU1}(y)}{2}\right)\right)
	\end{equation}
	for all $(x,y) \in \mathbb{R}^2$. 
\end{lemma}
\begin{proof}
	According to the definition of $g_2(x)$ in \eqref{func:glinduction} (or \eqref{eq:defg2_2}) and ${\rm ReLU1}(x)$ in \eqref{eq:defReLU1},
	we have $ \varphi(x,y) \in {\mathcal N}_2^{15}$ and
	\begin{equation}\label{key}
		\varphi(x,y) = 0
	\end{equation}
	for any $(x,y) \notin [0,1]^2$ as ${\rm ReLU1}(x)$ and ${\rm ReLU1}(y)$ will equal to $0$ or $1$. 
	Then, \eqref{eq:b-DNN2} holds, given that ${\rm ReLU1}(x) = x$ and ${\rm ReLU1}(y) = y$ for $(x,y) \in [0,1]^2$.
\end{proof}

By employing the above lemma, we can construct the following theorem pertaining to the representation
of linear finite element functions by using ReLU DNNs with only two hidden layers on
two-dimensional space.
\begin{theorem}\label{thm:FEM-DNN}
	Assume $u_h$ is a two-dimensional linear finite element function, which can be written as
	\begin{equation}\label{key}
		u_h(x,y) = \sum_{i=1}^N \mu_i \varphi(T_i(x,y)), 
	\end{equation}
	where $T_i : \mathbb{R}^2 \mapsto \mathbb{R}^2$ is an affine mapping and $N$ denotes
	the number of degree of freedom. Then, $u_h(x)$ 
	can be reproduced globally by a ReLU DNN with only two hidden layers and $15N$ neurons
	at most for each layer, i.e., $u_h(x,y)  \in {\mathcal N}_2^{15N}$.
\end{theorem}
According to~\cite{arora2018understanding}, any d-dimensional continuous piecewise
linear function can be represented by ReLU DNNs with at most $\lceil \log_2(d+1) \rceil$ hidden layers.
However, the representation theory in \cite{arora2018understanding} requires an extremely complicated
construction by induction with a tremendously high number of neurons. On the other hand, \cite{he2020relu}
shows that ReLU DNNs with only one hidden layer cannot be used to represent general linear
finite element functions even on two-dimensional space. 
Thus, the focus inevitably turns to finding
a concise and explicit representation of linear finite element functions using ReLU DNNs with 
only two hidden layers on two-dimensional space. 
Here, Theorem~\ref{thm:FEM-DNN} provides the answer for any two-dimensional linear finite element functions on a uniform mesh or meshes, which can be obtained by adding affine mappings.

\section{Concluding remarks}\label{sec:conclusion}
By carefully studying the hierarchical representation for $s(x) = x^2$ and $m(x,y) = xy$,
we finally establish a new understanding, interpretation, and extension
for the approximation results of ReLU DNNs for these two functions, which play a critically important
role in a series of recent approximation results of ReLU DNNs. 
These discoveries provide some precise and nontrivial connections between ReLU DNNs and
finite element functions especially for the hierarchical basis method. 
By applying these connections directly, we obtained the tightest error bound for 
approximating polynomials with different norms. 
We also showed that ReLU DNNs with this special structure can 
be applied only to approximate quadratic functions at an exponential rate. 
Furthermore, we unexpectedly developed an elegant and explicit formula whereby ReLU DNNs with 
two hidden layers can reproduce the basis function of the two-dimensional finite
element functions on the uniform grid, which happens to be the minimal 
number of layers needed to represent a locally supported piecewise linear function on $\mathbb{R}^2$ 
as discussed in \cite{he2020relu}.

The connections between hierarchical basis and the ReLU DNNs open up some promising
directions for the application of deeper and richer mathematical structures in finite element
or other classical approximation methods in relation to ReLU DNNs. 
For example, it is well worth extending the explicit representation of finite element
functions using ReLU DNNs on high-dimensional and unstructured meshes.
A related and equally intriguing problem is that of determining how to apply ReLU DNNs in
numerical solutions for partial differential equations~\cite{lu2021deepxde} based on these connections with finite elements.

\section*{Acknowledgements}
This work was partially supported by the Center for Computational Mathematics and Applications (CCMA) at The Pennsylvania State University, the Verne M. William Professorship Fund 
from The Pennsylvania State University, and the National Science Foundation (Grant No. DMS-1819157).

\bibliographystyle{plain}
\bibliography{dnn_hierarchy.bib}

\end{document}